\documentclass[psamsfonts]{amsart}
\usepackage{amssymb}
\usepackage{dsfont}
\usepackage[all]{xy}
\usepackage[mathcal]{eucal}
\usepackage{verbatim}
\usepackage{cite}
\usepackage{units}
\usepackage{amsmath,amsthm}

\DeclareMathOperator{\Aut}{Aut}

\newcommand{\col}{\colon}

\newcommand{\disp}{\displaystyle}

\newcommand{\scC}{\mathcal{C}}

\providecommand{\To}{\longrightarrow}

\providecommand{\cal}[1]{\mathcal{#1}}

\providecommand{\sma}{\wedge}
\newcommand{\n}{{\underline{n}}}
\newcommand{\inv}{{-1}}

\newcommand{\spectra}{\mathcal{S}}

\newcommand{\wre}{\mathord{\wr}}

\newtheorem{thrm}{Theorem}

\newtheorem{lemma}[thrm]{Lemma}

\theoremstyle{definition}
\newtheorem{defn}[thrm]{Definition}

\theoremstyle{remark}

\newtheorem{ex}{Example}
\newtheorem{remark}[thrm]{Remark}

\makeatletter
\let\c@equation\c@thrm
\makeatother

\begin{document}

\title{A comparison of norm maps}
\author[Anna Marie Bohmann]{Anna Marie Bohmann\\
with appendix by Anna Marie Bohmann and Emily Riehl}

\subjclass[2010]{Primary 55P91, 55P42; Secondary 18D30}

\hyphenation{Grothen-dieck}

\begin{abstract}We present a spectrum-level version of the norm map in equivariant homotopy theory based on the algebraic construction in \cite{GM1997}.  We show that this new norm map is same as the construction in \cite{HHR2009}.  Our comparison of the two norm maps gives a conceptual understanding of the choices inherent in the definition of the multiplicative norm map.
\end{abstract}
\maketitle

This paper serves to contextualize and explain the construction of the norm maps in equivariant stable homotopy theory.  The norm map should be thought of as a sort of multiplicative induction functor that takes $H$--equivariant spectra to $G$--equivariant spectra, where $H$ is a finite index subgroup of a compact Lie group $G$.  We compare the definition of the norm map used in the recent solution to the Kervaire invariant one problem \cite{HHR2009} to the norm map implicit in earlier work of Greenlees and May \cite{GM1997}.  Greenlees and May define an algebraic norm map on homotopy groups, whereas Hill, Hopkins and Ravenel have a spectrum-level construction that agrees with the Greenlees--May construction on the algebraic level.  We give a spectrum-level version of the Greenlees--May construction and compare the two constructions at the spectrum level. Our comparison gives a new conceptual understanding of the choices inherent in the definition of the multiplicative norm  map.

  First, note that we will use orthogonal spectra and push all of the choice of universe issues into choice of model structure.  It follows from \cite[Theorem V.1.7]{MM2002} that we can thus get away with thinking of $G$--spectra as simply $G$--objects in the category of spectra.  In turn, we will think of these as covariant functors from the one-object category $G$ into the category $\cal{S}$ of spectra.  

Let $G$ be a compact Lie group and $H< G$ be a subgroup such that $[G:H]=n$.  Let $B_{G/H}G$ be the translation groupoid of $G$ acting on $G/H$. That is, $B_{G/H}G$ has objects $x\in G/H$ and morphisms $x\stackrel{g}{\To}gx$ for all $g\in G$.
If we think of $H$ as a one-object category,  we have an inclusion of groupoids $\iota\col H\hookrightarrow B_{G/H}G$ given by sending the object $\ast_H$ of $H$ to the identity coset.  Since $B_{G/H}G$ is connected and the endomorphisms of $eH\in G/H$ are just $H$, this is an equivalence of categories.
\begin{defn}  \label{HHRdefn}
The Hill--Hopkins--Ravenel norm map is the following composite.
\[
\xymatrix{
\spectra^H \ar[r]^-{\simeq}\ar[dr]_-{N_H^G} &\spectra^{B_{G/H}G} \ar[d]^-{p_*^{\sma}}\\
 & \spectra^G }
\]
\end{defn}
Here the map $\spectra^H\to \spectra^{B_{G/H}G}$ is given by choice of an inverse equivalence to the inclusion $\iota\colon H\to B_{G/H}G$.  The map $p_*^{\sma}$ is the ``monoidal pushforward'' of \cite[page 109]{HHR2009}.  For a functor $X\in \spectra^{B_{G/H}G}$, $p_*^\sma$ is defined at the object $\ast_G$ of $G$ by
\begin{equation} \label{HHRindexedobj}
(p_*^{\sma}X)(\ast_G)=\bigwedge_{x \in G/H}X(x).
\end{equation}
The image of a morphism $g\in G$ is given by the smash product of the images of  the morphisms in $B_{G/H}G$ induced by the action of $g$ on the elements of $G/H$.  If we denote by $g_x$ the map $x\stackrel{g}{\To} gx$ in $B_{G/H}G$, we can write $(p_*^\sma X)(g)$ explicitly as the map
\begin{equation} \label{HHRindexedmorph}
\bigwedge_{x\in G/H}X(x)\xrightarrow{\sma X(g_x)} \bigwedge_{x\in G/H} X(gx).
\end{equation}

In \cite{GM1997}, Greenlees and May only define an ``algebraic'' version of the norm on equivariant homotopy or cohomology groups.  However, implicit in their discussion of algebraic norm maps is a definition of norm maps on the spectrum level.   This norm map again can be described as a composite of functors as follows.
\begin{defn} \label{GMdefn}
The spectrum-level version of the Greenlees--May norm map is the composite
\[\xymatrix{ \spectra^H \ar[r]^-{\sma^n} \ar[dr]_-{N_H^G} & \spectra^{\Sigma_n\wre H} \ar[d]^-{\alpha^*} \\
 & \spectra^G}
\]
\end{defn}
Here $\Sigma_n\wre H$ is the wreath product of $H$ and the symmetric group $\Sigma_n$, and the horizontal map corresponds to taking an $n$-fold smash power of an $H$-spectrum. We will be more explicit about this construction in Lemma \ref{defnindexedsmash}.  The vertical map $\alpha^*$ is the pullback along a homomorphism $\alpha\colon G\to \Sigma_n\wre H$.  Concretely, we get our homomorphism $\alpha$ by choosing a transversal to $H$ in $G$.  This construction will in fact be key in what follows, so we give details now.  A set of coset representatives  $\{t_1,\dotsc, t_n\}$ of $G/H$ defines a homomorphism $\alpha\colon G\to \Sigma_n\wre H$ by the formula
\begin{equation}\label{wreathhomomorphism}
 \alpha(g)=(\sigma_g,h_1(g),\dotsc,h_n(g))
\end{equation}
where $\sigma_g\in \Sigma_n$ and $h_i(g)\in H$ are defined by the equation $gt_i=t_{\sigma_g(i)}h_i(g)$. In other words, $\sigma_g$ is the permutation determined by the action of $g$ on the cosets of $G/H$, and the $h_i(g)$ give the $H$-action induced on each coset.  

\begin{remark}
The construction of Definition \ref{GMdefn}, with its use of the wreath product $\Sigma_n\wre H$, is nicely parallel to the construction of the norm map in group cohomology given by Evens \cite[Chapter 5]{Evens1991}.  In fact, our homomorphism $\alpha$ is his ``monomial representation'' and Evens's notion of the ``tensor induced module''---a monoidal induction from $H$--modules to $G$--modules---is precisely our definition of the Greenlees--May norm map where the category $\spectra$ is replaced by the category of modules over a ground ring.
\end{remark}

The heart of our comparison of the construction of the norm map in Definitions \ref{HHRdefn} and \ref{GMdefn} is the following diagram.
\begin{equation}\label{maindiagram}
\xymatrix{
\spectra^H \ar[d]_{c} \ar[dr]^-{\kappa^*}_-{\simeq}\\
\spectra^{B_{\n}\Sigma_n\times H} \ar[d]_{\n^\sma}\ar[r]^{\beta^*} & \spectra^{B_{G/H}G}\ar[d]^{p_*^\sma}\\
\spectra^{\Sigma_n\wre H}\ar[r]^{\alpha^*} &\spectra^G
}
\end{equation}

In this diagram the left vertical composite is a factorization of the $n$-fold smash product construction in the Greenlees--May norm map, so that the composite down the left side is the Greenlees--May norm and the composite along the right side is the Hill--Hopkins--Ravenel norm.  We think of the ``norm'' construction as having two parts: ``choosing a transversal'' and ``indexed smash product.'' By ``indexed smash product,'' we mean a construction defined by an (unordered) $n$-fold smash product, as in equations (\ref{HHRindexedobj}) and (\ref{HHRindexedmorph}).  This concept is formalized in Appendix \ref{sec:unify-fram-index}.  In Diagram (\ref{maindiagram}), all the maps from the left column to the right can be specified by choosing a transversal, whereas the lower two vertical maps are forms of indexed smash products.  Thus, the Greenlees--May norm is given by taking an indexed smash product and then choosing a transversal, whereas the Hill--Hopkins--Ravenel norm map is given by first choosing a transversal and then taking indexed smash products.  The fact that these two constructions agree---that this diagram commutes---essentially says that one can perform these constructions in either order.  We devote the rest of this section to proving this compatibility.

\begin{thrm}\label{maintheorem} For compatible choices of homomorphism $\alpha\col G\to \Sigma_n\wre H$ and equivalence of categories $B_{G/H}G\to H$, Diagram (\ref{maindiagram}) commutes.  Thus the Hill--Hopkins--Ravenel norm map and the Greenlees--May norm maps agree.
\end{thrm}

The proof follows from the following two lemmas.

\begin{lemma}\label{triangledefns} A choice of transversal $\{t_1,\dotsc,t_n\}$ of $H$ in $G$ determines functors $\kappa\col B_{G/H}G\to H$ and $\beta\col B_{G/H}G\to B_{\n}\Sigma_n\times H$ such that the diagram
\[
\xymatrix{ H&\\
B_\n\Sigma_n\times H \ar[u] &B_{G/H}G\ar[l]_{\beta} \ar[ul]_{\kappa}
}
\]
commutes.
\end{lemma}
\begin{proof}
Fix a transversal $\{t_1,\dotsc,t_n\}$ of $H$ in $G$, where $t_1=e$.  This choice determines an inverse to the equivalence of categories $\iota\colon H\to B_{G/H}G$ described above. Explicitly, the inverse $\kappa\colon B_{G/H}G\to H$ is given by identifying $\Aut(t_iH,t_iH)$ with $H$ via $t_i^{-1}Ht_i$.  The transversal also gives a convenient shorthand for labeling the arrows in the $B_{G/H}G$: we henceforth use $g_i\col t_iH\to t_{\sigma_g(i)}H$ to denote the map $g_{t_iH}\col t_i H\to t_{\sigma_g(i)} H$. With this notation, a map $g_{i}\col t_iH\to t_{\sigma_g(i)}H$ in $B_{G/H}G$ is mapped under the functor $\kappa$ to $h_i(g)\colon \ast_H\to \ast_H$. Here we use $\ast_H$ to denote the single object of the category $H$. 

Our choice of transversal also allows us to define the functor $\beta$.  On objects, $\beta(t_iH)=(i,\ast_H)=i$, and on a morphism $g_i\colon t_iH\to t_{\sigma_g(i)}H$ we define $\beta(g_i)=({\sigma_g}_{i},h_i(g))\colon i\to \sigma_g(i)$. We next check that $\beta$ is functorial. Given $g$ and $\gamma$ in $G$, we must check that the diagram
\[
\xymatrix{ \beta(t_iH)\ar[r]^{\beta(g_i)} \ar[dr]_{\beta(\gamma g_i)}& \beta(t_{\sigma_g(i)})\ar[d]^{\beta(\gamma_{\sigma_g(i)})}\\
&\beta(t_{\sigma_{\gamma g}(i)}H)}
\]
commutes in $B_{\n}\Sigma_n$.  The diagonal arrow is the map $(\sigma_{\gamma g}, h_i(\gamma g))\colon i\to \sigma_{\gamma g}(i)$, while the composite around the top right is 
\[(\sigma_\gamma,h_{\sigma_g(i)}(\gamma))\circ(\sigma_g,h_i(g))=(\sigma_\gamma\sigma_g,h_{\sigma_g(i)}(\gamma)h_i(g)).\]
  Since $\sigma_{\gamma g}$ is the permutation given by the action of $\gamma g$, we see that $\sigma_\gamma\circ \sigma_g=\sigma_{\gamma g}$.  Furthermore, by definition $h_i(\gamma g)$ is the  element of $H$ such that $\gamma g t_i=t_{\sigma_{\gamma g}}h_i(\gamma g)$.  Thus we have equalities 
\begin{align*}
\gamma g t_i &=\gamma t_{\sigma_g(i)}h_i(g)\\
&=t_{\sigma_\gamma(\sigma_g(i))}h_{\sigma_g(i)}(\gamma)h_i(g)\\
&=t_{\sigma_{\gamma g}(i)}h_{\sigma_g(i)}(\gamma)h_i(g).
\end{align*}
which prove that  $\beta$ is functorial.

Commutativity of the triangle 
\[
\xymatrix{ H&\\
B_\n\Sigma_n\times H \ar[u] &B_{G/H}G\ar[l]_{\beta} \ar[ul]_{\kappa}
}
\]
is verified directly from the definitions of $\beta$ and $\kappa$.  Note that the vertical map here is just the projection of $B_\n\Sigma_n$ to the terminal category.
\end{proof}

From Lemma \ref{triangledefns}, we obtain the functors $\kappa^*$ and $\beta^*$ in Diagram (\ref{maindiagram}) by precomposition.  The following lemma tells us how to define the remaining functor of Diagram~(\ref{maindiagram}).

\begin{lemma}\label{defnindexedsmash} There exists an indexed smash product functor $\n^\sma\col \spectra^{B_{\n}\Sigma_n\times H}\to \spectra^{\Sigma_n\wre H}$ whose value on an object $X\in \spectra^{B_\n\Sigma_n\times H}$ is the functor $\n^\sma X\col \Sigma_n\wre H\to \spectra$ defined by the equations
\begin{align*}
(\n^\sma X)(\ast_{\Sigma_n\wre H})&=\bigwedge_{i\in \n}X(i)\\
(\n^\sma X)(\sigma,h_1,\dotsc, h_n)&=\bigwedge_{i\in \n}X(\sigma_i,h_i)\colon \bigwedge_{i\in \n}X(i)\to \bigwedge_{i\in \n}X(\sigma(i)).
\end{align*}
\end{lemma}
We postpone the proof of Lemma \ref{defnindexedsmash} until after the proof of Theorem \ref{maintheorem}.

\begin{proof}[Proof of {Theorem \ref{maintheorem}}]
Fix a transversal $\{t_1,\dotsc,t_n\}$ of $H$ in $G$.  By Lemma \ref{triangledefns}, this determines functors $\kappa$ and $\beta$.  This choice also determines the homomorphism $\alpha\col G\to \Sigma_n\wre H$ in Equation (\ref{wreathhomomorphism}).  Precomposition with these maps determines the functors $\kappa^*$, $\beta^*$ and $\alpha^*$ in Diagram~(\ref{maindiagram}).  The commutativity of the triangle in Lemma \ref{triangledefns} thus implies commutativity of the upper triangle in Diagram (\ref{maindiagram}).

We now turn to the lower square in Diagram (\ref{maindiagram}). 
 We have already described the functor $p_*^\sma$ in Equations (\ref{HHRindexedobj}) and (\ref{HHRindexedmorph}), and it makes sense to think of $p_*^\sma$ as a smash product ``indexed over $G/H$.''  Lemma \ref{defnindexedsmash} defines the functor $\n^\sma$, which has the form of a ``smash product indexed over $\n$.''   We explicitly check that the square in question commutes.

Given $X\in \spectra^{B_{\n}\Sigma_n\times H}$, the functor $p_*^\sma(\beta^*X)\in \spectra^G$ is given at the object $\ast_G$ by 
\[p_*^\sma(\beta^* X)(\ast_G)=\bigwedge_{t_iH\in G/H}\beta^*(X)(t_iH)= \bigwedge_{t_iH\in G/H} X(\beta(t_iH))= \bigwedge_{i\in \n} X(i).\]
On the other hand, $\alpha^*(\n^\sma(X))$ is given at $\ast_{G}$ by
\[ \alpha^*(\n^\sma X)(\ast_G)= (\n^\sma X)(\alpha(\ast_G))= (\n^\sma X)(\ast_{\Sigma_n\wre H})\cong \bigwedge_{i\in\n}X(i).\]
Hence $p_*^\sma(\beta^*X)$ and $\alpha^*(\n^\sma X)$ are the same on objects. 

Unraveling the definitions on morphisms, we see that for $g\in G$, the map $p_*^\sma(\beta^*X)(g)$ is 
\[
\xymatrix@C=.3\linewidth{ \disp \bigwedge_{t_iH\in G/H} (\beta^*X)(t_iH)\ar[r]^{\bigwedge (\beta^*X)(g_i)} \ar@{=}[d] & \disp \bigwedge_{t_iH\in G/H}(\beta^* X)(t_{\sigma_g(i)}H) \ar@{=}[d]\\
\disp \bigwedge_{t_iH\in G/H} X(\beta(t_iH))\ar[r]^{\bigwedge X(\beta(g_i))} \ar@{=}[d] &\disp \bigwedge_{t_iH\in G/H}X(\beta(t_{\sigma_g(i)}H))\ar@{=}[d]\\
\disp \bigwedge_{i\in \n}X(i) \ar[r]^{\bigwedge X({\sigma_g}_i,h_i(g))} &\disp \bigwedge_{i\in \n} X(\sigma_g(i)).
}\]
Meanwhile, $(\alpha^*\n^\sma(X))(g)$ is
\[
\xymatrix@C=.3\linewidth{ 
\disp \n^\sma X(\alpha(\ast_G))\ar[r]^{\n^\sma X(\alpha(g))} \ar@{=}[d]& \disp\n^\sma(X \alpha(\ast_G)) \ar@{=}[d]\\
\n^{\sma}X(\ast_{\Sigma_n\wre H})\ar[r]^{\n^\sma X(\sigma_g,h_i(g),\dotsc,h_n(g))} \ar@{=}[d]& \n^\sma X(\ast_{\Sigma_n\wre H}) \ar@{=}[d]\\
\disp \bigwedge_{i\in \n}X(i)\ar[r]^{\bigwedge X({\sigma_g}_i,h_i(g))} & \disp \bigwedge_{i\in \n}X(\sigma_g(i)).
}
\]
Thus we see that the lower square of Diagram (\ref{maindiagram}) commutes, which completes the proof that the entirety of Diagram (\ref{maindiagram}) commutes.  Hence, given choices of homomorphism $\alpha$ and inverse equivalence $\kappa$ coming from a transversal of $H<G$, the two constructions of the norm map agree. 
\end{proof}

\begin{remark}
One way to conceptualize this result is as saying that in order to have norm map, one must fix a transversal and take an indexed smash product, but it doesn't matter in which order one performs these constructions.
\end{remark}

We now turn to the technical business of proving that the functor $\n^\sma$ of Lemma \ref{defnindexedsmash} is well defined.  This functor can be constructed by hand, and we give this construction below.  A categorical description that unifies the construction of $p_*^\sma$ and $\n^\sma$ is given in Appendix \ref{sec:unify-fram-index}.
\begin{proof}[Proof of {Lemma \ref{defnindexedsmash}}]

Given an object $X\in \spectra^{B_{\n}\Sigma_n\times H}$, define $\n^\sma X$ to be the functor $\Sigma_n\wre H\to \spectra$ given by the equations
\begin{align*}
(\n^\sma X)(\ast_{\Sigma_n\wre H})&=\bigwedge_{i\in \n}X(i)\\
(\n^\sma X)(\sigma,h_1,\dotsc, h_n)&=\bigwedge_{i\in \n}X(\sigma_i,h_i)\colon \bigwedge_{i\in \n}X(i)\to \bigwedge_{i\in \n}X(\sigma(i)).
\end{align*}
We must check that this is actually a functor.  Since the product of $(\sigma,h_1,\dotsc,h_n)$ and $(\tau,k_1,\dotsc,k_n)$ in $\Sigma_n\wre H$ is $(\sigma\tau,h_{\tau(1)}k_1,\dotsc,h_{\tau(n)}k_n)$, this amounts to showing that the diagram
\[\xymatrix@C=.3\linewidth{
\disp \bigwedge_{i\in\n}X(i)\ar[r]^{\bigwedge X(\tau_i,k_i)} \ar[dr]_{\bigwedge X(\sigma\tau_i, h_{\tau(i)}k_i)\quad\quad\quad} &\disp \bigwedge_{i\in \n} X(\tau(i))\ar[d]^{\bigwedge X(\sigma_{\tau(i)}, h_{\tau(i)})}\\
&\disp \bigwedge_{i\in \n} X(\sigma(\tau(i)))
}\]
commutes.  This follows from the definition of the morphisms in $B_{\n}\Sigma_n$.  Hence $\n^\sma(X)$ is a well-defined object of $\spectra^{\Sigma_n\wre H}$.

We also need to show that $\n^\sma$ is functorial, ie that if $X\stackrel{f}{\To} Y$ is a natural transformation in $\spectra^{B_{\n}\Sigma_n\times H}$, we have a natural transformation $\n^\sma f\col \n^\sma X\to \n^\sma Y$ in $\spectra^{\Sigma_n\wre H}$.  We define the map $\n^\sma f\col \n^\sma X(\ast_{\Sigma_n\wre H})\to \n^\sma Y(\ast_{\Sigma_n\wre H})$ to be $\bigwedge_{i\in\n} f_i$, where $f_i$ denotes the component of $f$ at the object $i\in B_\n\Sigma_n\times H$.
 Naturality of $\n^{\sma}f$ is then just the commutativity of the following diagram:
\begin{equation}\label{naturality}
\xymatrix@C=.3\linewidth{
\disp
\bigwedge_{i\in \n}X(i) \ar[r]^-{\bigwedge X(\sigma_i, h_i)}\ar[d]_-{\bigwedge f_i} &\disp \bigwedge_{i\in \n} X(\sigma(i))\ar[d]^-{\bigwedge f_{\sigma(i)}}\\
\disp \bigwedge_{i\in \n} Y(i)\ar[r]_-{\bigwedge Y(\sigma_i, h_i)} & \disp\bigwedge_{i\in \n} Y(\sigma(i))
}
\end{equation}
Naturality of $f$ implies the commutativity of this square at any component $i$ of the smash product; the desired commutativity follows.
\end{proof}

\appendix

\section{A unifying framework for indexed tensor products}\label{sec:unify-fram-index}

Let $\mathrm{FinSet}_{\mathrm{\bf iso}}$ denote the category of finite sets and isomorphisms. For any small category $J$, the data of a functor $P \colon J \to \mathrm{FinSet}_{\mathrm{\bf iso}}$ can be encoded by a functor $p \colon I \to J$ of the form described below. Functors $p$ arising in this way are called \emph{finite covering categories}; in the categorical literature, it is more common to say that $p$ is both a \emph{discrete left fibration} and a \emph{discrete right fibration}.  In \ref{finitecovercat}, we give first a concise and then a concrete description of finite covering categories, and explain the Hill--Hopkins--Ravenel construction of the indexed tensor product associated to $p$. In \ref{grothendieck}, we give an alternate description of the Hill--Hopkins--Ravenel construction, implicit in the appendices of \cite{HHR2009}, that makes use of the so-called Grothendieck construction. By appealing to this categorical machinery, the proofs of several of the basic propositions concerning indexed tensor products are automatic. Finally, in \ref{generalindexed}, we show that the hypotheses on the functors $p \colon I \to J$ necessary to construct an indexed tensor product may be relaxed. In so doing, we unify the description of the norms of Greenlees--May and Hill--Hopkins--Ravenel, the main objective of the body of this paper.

\subsection{Indexed tensor products via finite covering categories}\label{finitecovercat}

Write ${\ast}$ for the terminal category with a single object and a single identity arrow and ${\bf 2}$ for the category with two objects $0,1$ and a single non-identity arrow $0 \to 1$.  A \emph{finite covering category} is a functor $p \colon I \rightarrow J$ such that for each commuting square, there are unique lifts
\[\vcenter{\xymatrix{ {\ast} \ar[d]_0 \ar[r]^i & I \ar[d]^p \\ {\bf 2} \ar[r]_f \ar@{-->}[ur]_{\exists !} & J }} \qquad \text{and} \qquad \vcenter{ \xymatrix{ {\ast} \ar[d]_1 \ar[r]^{i'} & I \ar[d]^p \\ {\bf 2} \ar[r]_f \ar@{-->}[ur]_{\exists !} & J}}\]
In other words, for each arrow $f$ of $J$ and specified lift $i$ of its domain, there is a unique arrow of $I$ with domain $i$ lifting $f$, and similarly, for each specified lift $i'$ of its codomain, there is a unique arrow of $I$ with codomain $i'$ lifting $f$.

Unpacking this definition, we see that the domain $I$ of a finite covering category can be described as follows.  Denote by $p^{\inv}(j)$ the fiber over the object $j\in J$, meaning the subcategory of $I$ mapping to the identity on $j$. Each of these fibers is discrete, and if $J$ is connected,
 all the fibers have the same cardinality, which we denote $n$.  This follows from understanding the arrows of $I$.  An arrow $f\col j\to j'$ has $n$ lifts to arrows from an object of $p^\inv(j)$ to an object of $p^\inv(j')$, and this collection of lifts determines a bijection between the fibers.  This correspondence is injective by the unique left lifting and surjective by the unique right lifting of arrows in a covering category.  No non-identity arrows can be in the fiber over identity arrows of $J$; hence the fibers of $p$ are \emph{discrete.}

Composition of arrows in $I$ is defined as in $J$. For example, the data illustrated in the following picture specifies a covering category in which the fibers have cardinality three.
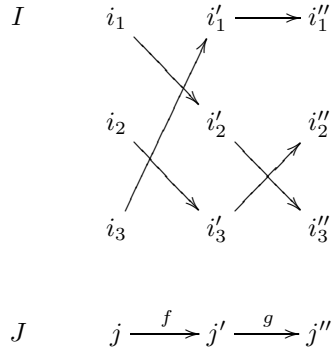
\[\xymatrix{ I & i_1 \ar[dr] & i_1' \ar[r] & i_1'' \\ & i_2 \ar[dr] & i_2' \ar[dr] & i_2'' \\ & i_3 \ar[uur] & i_3' \ar[ur] & i_3'' \\ J & j \ar[r]^f & j' \ar[r]^g & j''}\]

Given a symmetric monoidal category $(\scC,\otimes,\mathbf{1})$ and a finite covering category $p \colon I \rightarrow J$, the \emph{indexed tensor product} or \emph{monoidal pushforward} is a functor $p^{\otimes}_* \colon \scC^I \rightarrow \scC^J$. The image of  $X \colon I \rightarrow \scC$ is the functor $J \to \scC$ defined by 
 \[
\xymatrix{ j \ar[d]_f \ar@{|->}[r] & \displaystyle\bigotimes_{p^{-1}(j)} X(i) \ar[d]^{\otimes X(f_i)}  \\ j' \ar@{|->}[r]  & \displaystyle\bigotimes_{p^{-1}(j')} X(i')}\]

\begin{remark}\label{yaysymmetry} The categorically astute reader may be concerned to note that in the diagram above we have not specified the order in which the tensor products indexed on an (unordered) set are taken. Since we will only consider maps between these ``unordered'' tensor products that arise from simple tensors of maps, as in the right arrow above, the coherence and naturality of the isomorphism $X \otimes Y \cong Y \otimes X$ in our symmetric monoidal category ensures that everything is well-defined, at least up to isomorphism.
\end{remark}

This gives a concrete description of the indexed tensor product associated to a finite covering category. This construction also arises in a more categorical way using the Grothendieck construction, detailed in \ref{grothendieck}.  The covering category set up is crucial for this more categorical description, but is not in general necessary to define indexed tensor products, as explained in \ref{generalindexed}.

\subsection{Indexed tensor products via the Grothendieck construction}\label{grothendieck}

There is a functor \[ \mathrm{\bf FinSet}_{\mathrm{\bf iso}}^J \to \mathrm{\bf Cat}/J\] whose image consists of the finite covering categories and all maps over $J$. The construction of a finite category $p \colon I \to J$ from $P \colon J \to \mathrm{\bf FinSet}_{\mathrm{ \bf iso}}$ is sometimes called the \emph{Grothendieck construction}. 

Given $P$, we obtain a finite covering category $p\col I\to J$ where the fiber $p^{\inv}(j)$ is 
the set $Pj$ thought of as a discrete category.  An arrow $f\colon j\to j'$ has $n$ lifts to $I$ which are determined by the isomorphism $Pf\colon Pj\to Pj'$: namely, each object in $p^{\inv}(j)$ is the source of a unique lift of $f$, and each object in $p^{\inv}(j')$ is the target of a unique lift of $f$ so that the bijection set up by this correspondence is the isomorphism $Pf$ in $\mathrm{\bf FinSet}_{\mathrm{\bf iso}}$. Composition in $I$ is as in $J$, and there are no other morphisms in $I$.

In the literature, there is frequently confusion between this set-based Grothendieck construction and a categorified version that encodes functors $J \to \mathrm{\bf Cat}$ as fibration-like functors $I \to J$, whose fibers are typically not discrete.  Given $P \colon J \to \mathrm{\bf Cat}$, the categorified           
Grothendieck construction returns a category $I$ with the same set of objects: pairs 
$(j \in J, i \in Pj)$. But morphisms now have the form \[(f,h) \colon (j,i) \to (j',i')\] with $f \colon j \to j'$ in $J$ and
such that $h$ is an arrow in the category $Pj'$ from $(Pf)(i)$ to $i'$. 

\begin{ex} Let $\scC$ be any category and consider the constant functor $J \to \mathrm{\bf Cat}$ at $\scC$. The categorified Grothendieck construction returns the category $\scC \times J$ together with projection $\scC \times J \to J$. If we instead apply the set-based Grothendieck construction to the composite $J \to \mathrm{\bf Cat} \to \mathrm{\bf FinSet}_{\mathrm{\bf iso}}$ forgetting the arrows in $\scC$, then the result is $\mathrm{ob}\, \scC \times J \to J$.
\end{ex}

The categorified Grothendieck construction defines a functor \[ \mathrm{\bf Cat}^J \to \mathrm{\bf Cat}/J.\]  The image is the subcategory of \emph{opfibrations} over $J$ and cartesian morphisms. More details can be found in \cite[Volume 1 Section B.1.]{Johnstone2002} or \cite[Sections 2.1.1, 3.2.0]{Lurie2009}. Note that other sources such as \cite{MacMoer1992} only discuss the set-based version.

To construct the indexed tensor product, fix a symmetric monoidal category $(\scC, \otimes, \mathrm{\bf 1})$. Consider the following two functors $\mathrm{\bf FinSet}_{\mathrm{\bf iso}} \to \mathrm{\bf Cat}$:  the functor $I \mapsto \scC^I$ and the constant functor at $\scC$. When $\scC$ is symmetric monoidal, there is a natural transformation $\otimes^{-}$ from the former to the latter whose component     
at $I$ is the functor \[\otimes^I \colon \scC^I \to \scC.\]
Given a finite covering category $p$ encoded as a functor $P \colon J \to \mathrm{\bf FinSet}_{\mathrm{\bf iso}}$, we precompose the above natural transformation with $P$ to obtain a natural transformation whose target is constant functor at $\scC$. Applying the categorified Grothendieck construction to the morphism
\[\xymatrix@C=60pt@!0{J \ar@/^2.9ex/[r]^{j \mapsto \scC^{Pj}} \ar@/_2.9ex/[r]_-{\scC} \ar@{}[r]|{\big\Downarrow \otimes^{-}} & \mathrm{\bf Cat}}\]
 we obtain functors \[ \xymatrix{ \scC_P \ar[rr] \ar[dr] & & \scC \times J \ar[dl] \\ & J  }\] 
It follows from the preceding description of the Grothendieck construction that the category of sections for the opfibration $\scC_P \to J$ is $\scC^I$. It is even easier to see that the category of sections for $\scC \times J \to J$ is $\scC^J$. Tracing through the definitions shows that the induced functor $\scC^I\to \scC^J$ between the categories of sections is the indexed tensor product $p^\otimes_*$. See the appendices to \cite{HHR2009} for further details.

\subsection{Indexed tensor products in general}\label{generalindexed}

Suppose we're given a connected category $J$ and an arbitrary category $I$. To construct an indexed tensor product $\scC^I \rightarrow \scC^J$ of ``dimension'' $n$ (meaning the objects in the image of the functors $J \rightarrow \scC$ are all $n$-fold tensors in $\scC$), we need the following data:
\begin{itemize} 
\item for each $j \in J$, a set $P_j$ of $n$ distinct objects of $I$ 
\item for each $f \colon j \rightarrow j'$, a set $P_f$ of $n$ (necessarily distinct) arrows of $I$ specifying a ``directed bijection'' between $P_j$ and $P_{j'}$, ie each object of $P_j$ and $P_{j'}$ is respectively the domain or codomain of a unique arrow in this collection 
\item the sets $P_f$ should depend functorially on $J$, so that $P_{gf}$ is the set consisting of composites of the arrows in $P_g$ and $P_f$. By our condition on the object sets, there is only one possible way in which these composites can be defined. \end{itemize}

When $p \colon I \rightarrow J$ is a finite covering category, these sets are defined by the functor $P \colon J \to \mathrm{\bf FinSet}_{\mathrm{\bf iso}}$, or alternatively by taking lifts of objects and arrows of $J$ along $p$. But the point is that the category $I$ could in principle be much larger, having more objects and/or arrows. In such cases, it is unlikely that there will be a functor $I \rightarrow J$, much less a covering category. More surprisingly, $I$ can also be too small to be a covering category, as is the case in the construction of the functor $\n^{\sma}$ in Lemma \ref{defnindexedsmash}.

In general, the data needed for an indexed tensor product $\scC^I \to \scC^J$ comes from a functor of the form described below. Let $I^n$ denote the $n$-fold cartesian product of $I$ and let $I^n \backslash \Delta$ be the full subcategory on the complement of the fat diagonal. That is, objects of $I^n \backslash \Delta$ are ordered $n$-tuples of distinct objects of $I$ and morphisms are ordered $n$-tuples of morphisms, necessarily distinct as well. The group $\Sigma_n$ acts naturally on $I^n$ and this action descends to $I^n \backslash \Delta$. Write $(I^n \backslash \Delta) / \Sigma_n$ for the quotient in $\mathbf{Cat}$ of $I^n \backslash \Delta$ by this action.

\begin{defn} 
Given categories $J$ and $I$ equipped with a functor \[ P \colon J \rightarrow (I^n \backslash \Delta )/ \Sigma_n\] the \emph{indexed tensor product} is the functor $P^\otimes \colon \scC^I \to \scC^J$ defined as follows: An object $X \in \scC^I$ maps to the functor $J \to \scC$ defined by 
 \[
\xymatrix{ j \ar[d]_f \ar@{|->}[r] & \displaystyle\bigotimes_{i \in Pj} X(i) \ar[d]^{\displaystyle\bigotimes_{h \in Pf} X(h)}  \\ j' \ar@{|->}[r]  & \displaystyle\bigotimes_{i' \in P{j'}} X(i')}\]
Functoriality of $P$ implies that $P^\otimes(X)$ is an object of $\scC^J$; the verification that this assignment is functorial in $X$ is precisely that of Equation (\ref{naturality}).
\end{defn}

As in Remark~\ref{yaysymmetry},  we only define the functors in $\scC^J$ up to natural isomorphism. A strict definition requires a choice of ordering for the elements in each of our index sets. This is equivalent to specifying a lift of the functor $P$ through the canonical quotient arrow $I^n \backslash \Delta \rightarrow (I^n \backslash \Delta)/\Sigma_n$ from the colimit cone. 

\begin{ex} Any finite covering category $I \to J$ over a connected category $J$ with fibers of cardinality $n$ defines such a $P \colon J \to (I^n \backslash \Delta)/\Sigma_n$ by sending an object $j$ to the (unordered) set of its preimages. In particular, the construction $p^\wedge_*$ of Equations (\ref{HHRindexedobj}) and (\ref{HHRindexedmorph}),and more generally, the monoidal pushforward of \cite{HHR2009}, is an example of the indexed tensor product as just defined.
\end{ex}

\begin{ex} The functor $\n^\wedge$ of Lemma \ref{defnindexedsmash} is an indexed tensor product as well. In this case, the functor $P \colon \Sigma_n \wre H \to (B_\n\Sigma_n \times H)^n \backslash \Delta$ is given by \begin{align*} P(\ast_{\Sigma_n \wre H}) &= \{1,\ldots, n\} \\ P(\sigma, h_1,\ldots, h_n) &= \{ (\sigma_1,h_1),\ldots, (\sigma_n,h_n)\}. \end{align*} This example cannot arise from a covering category because there are simply not enough morphisms in $B_\n\Sigma_n$ to cover all the morphisms in $\Sigma_n \wre H$. In particular, suppose  $(\sigma, h_1,\ldots, h_n)$ and $(\sigma, h_1',\ldots, h_n')$ are distinct in $\Sigma_n \wre H$. If $h_i = h_i' = h$ for some $i$, then the morphism $(\sigma_i,h)$ covers them both.
\end{ex}

\bibliography{generalrefs}
\bibliographystyle{amsplain}

\end{document}